\documentclass[reqno,11pt]{amsart}
\usepackage{amssymb}
\usepackage{mathrsfs}
\usepackage{}
\usepackage{amsmath,amssymb}
\usepackage{paralist}
\usepackage{graphics} %% add this and next lines if pictures should be in esp format
\usepackage{epsfig} %For pictures: screened artwork should be set up with an 85 or 100 line screen
\usepackage[colorlinks=true]{hyperref}
\hypersetup{urlcolor=red, citecolor=blue}
\usepackage[top=1in,bottom=1in,left=1in,right=1in]{geometry}
\usepackage{cite}
\newtheorem{thm}{Theorem}[section]

\newtheorem{lem}[thm]{Lemma}

\theoremstyle{definition}

\newtheorem{rem}{Remark}[section]
\numberwithin{equation}{section}
 \newcommand{\be}{\begin{equation}}
 \newcommand{\ee}{\end{equation}}
 \newcommand\bes{\begin{eqnarray}}
 \newcommand\ees{\end{eqnarray}}
 \newcommand{\bess}{\begin{eqnarray*}}
 \newcommand{\eess}{\end{eqnarray*}}
 
\begin{document}
\title[Boundedness in a Keller-Segel system with logistic source]
      {Boundedness in a quasilinear fully parabolic Keller-Segel system with logistic source}%
\author[Zhang]{Qingshan Zhang }%
\address{Department of Mathematics, Southeast University, Nanjing 211189, P. R. China}
\email{qingshan11@yeah.net}

%%%%P13 address changed (we have some new administrative rules)

\author[Li]{Yuxiang Li }%
\address{Department of Mathematics, Southeast University, Nanjing 211189, P. R. China}
\email{lieyx@seu.edu.cn}

\thanks{Supported in part by National Natural Science Foundation of China (No. 11171063).}


%\date{}%
%\dedicatory{}%
%\commby{}%
% ----------------------------------------------------------------

\begin{abstract}
This paper deals with the Neumann boundary value problem for the system
\begin{eqnarray*}
  \left\{\begin{array}{lll}
     \medskip
     u_t=\nabla\cdot\left(D(u)\nabla u\right)-\nabla\cdot\left(S(u)\nabla v\right)+f(u),&{} x\in\Omega,\ t>0,\\
     \medskip
     v_t=\Delta v-v+u,&{}x\in\Omega,\ t>0
  \end{array}\right.
\end{eqnarray*}
in a smooth bounded domain $\Omega\subset\mathbb{R}^n$ $(n\geq1)$, where the functions $D(u)$ and $S(u)$ are supposed to be smooth satisfying
$D(u)\geq Mu^{-\alpha}$ and $S(u)\leq Mu^{\beta}$ with $M>0$, $\alpha\in\mathbb{R}$ and $\beta\in\mathbb{R}$ for all $u\geq1$, and the logistic source $f(u)$ is smooth fulfilling $f(0)\geq0$ as well as $f(u)\leq a-\mu u^{\gamma}$ with $a\geq0$, $\mu>0$ and $\gamma\geq1$ for all $u\geq0$.
It is shown that if
\begin{eqnarray*}
  \alpha+2\beta<\left\{\begin{array}{lll}
     \medskip
     \gamma-1+\frac{2}{n},&{} \mbox{for}\ 1\leq\gamma<2,\\
     \medskip
     \gamma-1+\frac{4}{n+2},&{} \mbox{for}\ \gamma\geq2,
  \end{array}\right.
\end{eqnarray*}
then for sufficiently smooth initial data the problem possesses a unique global classical solution which is uniformly bounded.
\end{abstract}
\maketitle
% ----------------------------------------------------------------

\section{Introduction}
In this paper we consider the initial-boundary value problem for the parabolic-parabolic quasilinear chemotaxis system with logistic source
\begin{eqnarray}\label{KS}
  \left\{\begin{array}{lll}
     \medskip
     u_t=\nabla\cdot\left(D(u)\nabla u\right)-\nabla\cdot\left(S(u)\nabla v\right)+f(u),&{} x\in\Omega,\ t>0,\\
     \medskip
     v_t=\Delta v-v+u,&{}x\in\Omega,\ t>0,\\
     \medskip
     \frac{\partial u}{\partial \nu}=\frac{\partial v}{\partial \nu}=0,&{} x\in\partial\Omega, t>0,\\
     \medskip
     u(x,0)=u_0(x),\quad v(x,0)=v_0(x),&{}x\in\Omega
  \end{array}\right.
\end{eqnarray}
in a bounded domain $\Omega\subset\mathbb{R}^n$ $(n\geq1)$ with smooth boundary $\partial\Omega$, where $u=u(x,t)$ denotes the density of bacteria and $v=v(x,t)$ is the concentration of oxygen. $\frac{\partial}{\partial \nu}$ represents differentiation with respect to the outward normal $\nu$ on $\partial\Omega$. The initial data $u_0\in C^0(\bar{\Omega})$ and $v_0\in W^{1,\theta}(\Omega)$ with $\theta>\max\{2, n\}$ are nonnegative functions. The parameter functions $D(u)$, $S(u)$ with $S(0)=0$ from $C^2([0, \infty))$ are supposed that
\begin{equation}\label{D}
D(u)\geq M_1(u+1)^{-\alpha}\quad\mbox{for all}\ u\geq0
\end{equation}
with $M_1>0$ and $\alpha\in\mathbb{R}$,
\begin{equation}\label{S}
S(u)\leq M_2(u+1)^{\beta}\quad\mbox{for all}\ u\geq0
\end{equation}
with $M_2>0$ and $\beta\in\mathbb{R}$, as well as $f(u)$ is smooth satisfying $f(0)\geq0$ and
\begin{equation}\label{logistic source}
f(u)\leq a-\mu u^{\gamma}\quad\mbox{for all}\ u\geq0
\end{equation}
with $a\geq0$, $\mu>0$ and $\gamma\geq1$.

\vskip 3mm

In 1970, Keller and Segel \cite{KS} proposed the chemotaxis system (\ref{KS}) to describe the biased movement of biological cell in response to chemical gradients. Since then the model has attracted significant interest in mathematical biology, and one of the main issues is under what conditions the solutions of (\ref{KS}) blow up or exist globally.

When $D(u)=1$, $S(u)=\chi u$ and $f(u)\equiv0$, system (\ref{KS}) corresponds to the so-called minimal model, which has been extensively studied. It proved that the solutions never blow up if $n=1$ \cite{Osaki-FE-2001}. In the two-dimensional case, if $\int_{\Omega}u_0<4\pi$ the solutions are global and bounded \cite{Nagai-FE-1997}, whereas $\int_{\Omega}u_0>4\pi$ and in the case $n\geq3$ the solutions blow up in finite time \cite{Horstmann-JAM-2001,Winkler-JMPA-2013}. In many applications the blow-up phenomena is an extreme case, so a logistic growth restriction of type (\ref{logistic source}) in (\ref{KS}) is expected to rule out the possible of blow-up for solutions. When $D(u)=1$, $S(u)=\chi u$ and $f(u)\leq a-\mu u^{2}$ in the model (\ref{KS}), all solutions are global and bounded provided that $n\leq2$ or $\mu>\mu_0$ with some $\mu_0>0$ in higher dimensions $n\geq3$ \cite{Osaki-NA-2002,Osaki-FE-2001,Winkler-CPDE-2010}. In the special case $f(u)=u-\mu u^2$, whenever $\frac{\mu}{\chi}$ is suitably large the solution $(u,v)$ stabilizes to the spatially homogeneous steady state $(\frac{1}{\mu}, \frac{1}{\mu})$ as $t\rightarrow\infty$ \cite{Winkler-JDE-2014}. Moreover, when $n\geq3$ there exists at least one global weak solution for any $\mu>0$ \cite{Lankeit-JDE-2015}.
However, it is unclear whether in higher dimensions $n\geq3$, the logistic source $f(u)$ with $\gamma=2$ in the problem (\ref{KS}) might be sufficient to rule out blow-up for arbitrarily small $\mu>0$ \cite{Winkler-CPDE-2010}.

Superlinear logistic growth are not always rule out chemotactic collapse in the Keller-Segel model. The initial-boundary value problem for the related system
\begin{eqnarray*}
  \left\{\begin{array}{lll}
     \medskip
     u_t=\Delta u-\chi\nabla\cdot(u\nabla v)+\lambda u-\mu u^{\gamma},&{} x\in\Omega,\ t>0,\\
     \medskip
     0=\Delta v-m(t)+u,&{}x\in\Omega,\ t>0
  \end{array}\right.
\end{eqnarray*}
with $m(t):=\frac{1}{|\Omega|}\int_{\Omega}u(x,t)$ was considered in \cite{Winkler-JMAA-2011}. It was shown that if $\lambda\in(1, \frac{3}{2}+\frac{1}{2n-2})$ with $n\geq5$ there exist initial data such that the solutions blow up in finite time.

On the other hand, the volume-filling effect can also prevent blow-up \cite{Hillen&Painter-AAM-2001,Painter&Hillen-CAMQ-2002}. In the case $D(u)=1$ and $f(u)\equiv0$ in (\ref{KS}). Horstmann and Winkler \cite{Horstmann&Winkler-JDE-2005} proved that if $S(u)\leq K(u+1)^{\beta}$
with $\beta<\frac{2}{n}$ and some $K>0$ the solutions are global and bounded, while if $S(u)\geq K(u+1)^{\beta}$
with $\beta>\frac{2}{n}$ and some $K>0$ the solutions blow up in finite or infinite time.

As to the Neumann boundary value problem for the associated parabolic-elliptic system
\begin{eqnarray*}
  \left\{\begin{array}{lll}
     \medskip
     u_t=\nabla\cdot\left(D(u)\nabla u\right)-\nabla\cdot\left(S(u)\nabla v\right),&{} x\in\Omega,\ t>0,\\
     \medskip
     0=\Delta v-m(t)+u,&{}x\in\Omega,\ t>0
  \end{array}\right.
\end{eqnarray*}
with $D(u)\simeq u^{-\alpha}$ and $S(u)\simeq u^{\beta}$ as $u\simeq \infty$, it was proved that if $\alpha+\beta<\frac{2}{n}$ the solutions are global and bounded, whereas if $\alpha+\beta>\frac{2}{n}$ there exist solutions that become unbounded in finite time \cite{Djie}.

The model (\ref{KS}) with $f(u)\equiv0$ has also been extensively studied \cite{Tao&Winkler-JDE-2012,Ishida-JDE-2014,Winkler-MMAS-2010,Stinner-JDE-2015}. It was shown that
\begin{itemize}
  \item if $\frac{S(u)}{D(u)}\leq K(u+\varepsilon)^{\theta}$ with $\theta<\frac{2}{n}$ and $K>0$ for some $\varepsilon\geq0$ and for all $u>0$, then all solutions to (\ref{KS}) are global and uniformly bounded \cite{Tao&Winkler-JDE-2012,Ishida-JDE-2014};
  \item if $\frac{S(u)}{D(u)}\geq K(u+1)^{\theta}$ with $\theta>\frac{2}{n}$ $(n\geq2)$ and $K>0$ for all $u>0$, then for any $m>0$, (\ref{KS}) possesses finite-time blow-up solutions with the mass $\int_{\Omega}u_0=m$ \cite{Winkler-MMAS-2010}.
\end{itemize}
The exponent $\frac{2}{n}$ seems critical for the finite-time blow-up and global existence properties of (\ref{KS}) with $f(u)\equiv0$.

The fully parabolic Keller-Segel system with logistic source (\ref{KS}) was considered in the recent papers \cite{Cao-JMAA-2013,Wang-DCDS-2014,li2015}. The authors proved that when $\Omega\subset\mathbb{R}^n$ ($n\geq2$) is a bounded {\it{convex}} domain with smooth boundary, global bounded classical solutions exist provided that $D(u)$, $S(u)$ and $f(u)$ satisfy (\ref{D})-(\ref{logistic source}) with $\alpha+\beta\in(0,\frac{2}{n})$. In \cite{li2015}, the authors extended the result to the degenerate case on non-convex domain under the same condition. Connected to the later two results, it is a natural question to ask:
\begin{quote}
(Q1) {\it {What role does the logistic source play in the system (\ref{KS})?}}
\end{quote}
Although a partial answer to (Q1) show that for any choice of $\beta<1$ the logistic damping rule out the occurrence of blow-up for the related special case $f(u)\leq a-\mu u^{2}$ provided that $S$ satisfies the condition of algebraic growth \cite{Cao-JMAA-2013}, it still remain to analysis:
\begin{quote}
(Q2) {\it {Can we provide a explicit condition involving the nonlinear diffusion, nonlinear chemosensitivity and the logistic-growth source to ensure global bounded solutions in the system (\ref{KS}) ?}}
\end{quote}

\vskip 3mm

In the present paper, our purpose is to answer (Q1) and (Q2). Namely, we shall give a general condition on $\alpha$, $\beta$ and $\gamma$, which guarantees the global existence and boundedness of classical solutions to (\ref{KS}) in {\it{non-convex}} bounded domains. Our main result is stated as follows.
\begin{thm}\label{main result}
Suppose that $\Omega\subset\mathbb{R}^n$ ($n\geq1$) is a domain with smooth boundary. Let (\ref{D})-(\ref{logistic source}) hold with $M_1>0$, $M_2>0$, $\alpha\in\mathbb{R}$, $\beta\in\mathbb{R}$, $a\geq0$, $\mu>0$ and $\gamma\geq1$. If
\begin{eqnarray*}
  \alpha+2\beta<\left\{\begin{array}{lll}
     \medskip
     \gamma-1+\frac{2}{n},&{} \mbox{if}\ 1\leq\gamma<2,\\
     \medskip
     \gamma-1+\frac{4}{n+2},&{} \mbox{if}\ \gamma\geq2,
  \end{array}\right.
\end{eqnarray*}
then for any nonnegative $u_0\in C^0(\bar{\Omega})$ and $v_0\in W^{1,\theta}(\Omega)$ with $\theta>\max\{2, n\}$, the problem (\ref{KS}) admits a unique global bounded classical solution.
\end{thm}

\begin{rem}
(i) Theorem \ref{main result} provides the general admissible parameters for global boundedness in (\ref{KS}). The result is valid on the non-convex domain and without any further assumptions on the size of $a\geq0$ and $\mu>0$.

(ii) In the case $\beta\leq0$ or $\gamma>\beta+1$ for $\beta>0$ or $\gamma>\beta+1-\frac{2(n-2)}{n(n+2)}$ for $\beta\geq1+\frac{2(n-2)}{n(n+2)}$, our result extends the recent work \cite{Wang-DCDS-2014} and \cite{li2015} which assert boundedness under the condition $\alpha+\beta<\frac{2}{n}$.

(iii) For the most interesting classical model
\begin{eqnarray}\label{mini model}
  \left\{\begin{array}{lll}
     \medskip
     u_t=\Delta u-\chi\nabla\cdot(u\nabla v)+a-\mu u^{\gamma},&{} x\in\Omega,\ t>0,\\
     \medskip
     v_t=\Delta v-v+u,&{}x\in\Omega,\ t>0,
  \end{array}\right.
\end{eqnarray}
our result shows that if $\gamma>3-\frac{4}{n+2}$, the model (\ref{mini model}) possesses a unique global classical solution for arbitrarily small $\mu>0$ and any bounded smooth domain $\Omega\subset\mathbb{R}^n$. This improves the recent result \cite[R5]{xiang-JDE-2015}. In particular, we rule out a chemotactic collapse in model (\ref{mini model}) with a cubic growth source $f(u)=u(u-b)(1-u)$ ($0<b<\frac{1}{2}$) for any biological parameters and any bounded smooth domain $\Omega\subset\mathbb{R}^n$.
\end{rem}

%\vskip 3mm

%The rest of the paper is organized as follows. We present some preliminaries in Section 2. Theorem \ref{main result} will be proved in Section %3.

\section{Preliminaries}
To begin with, let us first state one result concerning local well-posedness of the problem (\ref{KS}) and its proof can be found in \cite{Tao&Winkler-JDE-2012,Wang-DCDS-2014,Winkler-CPDE-2010,li2015}.

\begin{lem}\label{local existence}
Let $D(u)$ and $S(u)$ satisfy (\ref{D}) and (\ref{S}), respectively. Suppose that $f(u)\in W^{1,\infty}_{loc}(\mathbb{R})$ satisfying $f(0)\geq0$, and that $u_0\in C^0(\bar{\Omega})$ and $v_0\in W^{1,\theta}(\Omega)$ with $\theta>\max\{2, n\}$ are nonnegative functions. Then there exist $T_{\max}\in(0, \infty]$ and a uniquely determined
pair $(u, v)$ of nonnegative functions
\begin{eqnarray*}
&&u\in C^0(\bar{\Omega}\times[0, T_{\max}))\cap C^{2,1}(\bar{\Omega}\times(0, T_{\max})),\\
&&v\in C^0(\bar{\Omega}\times[0, T_{\max}))\cap C^{2,1}(\bar{\Omega}\times(0, T_{\max}))\cap L_{loc}^{\infty}([0, T_{\max});W^{1,\theta}(\Omega))
\end{eqnarray*}
solves (\ref{KS}) in the classical sense. In addition, if $T_{\max}<\infty$, then
 \begin{equation*}
\|u(\cdot, t)\|_{L^{\infty}(\Omega)}+\|v(\cdot, t)\|_{W^{1,\theta}(\Omega)}\rightarrow\infty\quad\mbox{as}\ t\nearrow T_{\max}.
\end{equation*}
\end{lem}

We next give the following basic estimates in spatial Lebesgue spaces for  $u$ and $\nabla v$.

\begin{lem}\label{basic est.}
Suppose that $D$, $S$ and $f$ satisfy (\ref{D})-(\ref{logistic source}), respectively. Then there exists $C>0$ such that the solution of (\ref{KS}) fulfils
\begin{equation}\label{u1}
\|u(\cdot,t)\|_{L^1(\Omega)}\leq C\quad\mbox{for all}\ t\in(0, T_{\max})
\end{equation}
and
\begin{equation}\label{nabla vs}
\|\nabla v(\cdot,t)\|_{L^s(\Omega)}\leq C\quad\mbox{for all}\ t\in(0, T_{\max}),
\end{equation}
where $s\in[1,\frac{n}{n-1})$. If, in addition, $\gamma\geq2$, then we have
\begin{equation}\label{nabla v2}
\|\nabla v(\cdot,t)\|_{L^2(\Omega)}\leq C\quad\mbox{for all}\ t\in(0, T_{\max}).
\end{equation}
\end{lem}
\begin{proof}
Integrating the first equation in (\ref{KS}) and using (\ref{logistic source}) gives
\begin{equation}\label{du1}
\frac{d}{dt}\int_{\Omega}u=\int_{\Omega}f(u)\leq a|\Omega|-\mu\int_{\Omega}u^{\gamma}\quad\mbox{for all}\ t\in(0, T_{\max}).
\end{equation}
Since $\gamma\geq1$, we can find $C_1>0$ such that $u\leq u^{\gamma}+C_1$ and thus derive that
\begin{equation*}
\frac{d}{dt}\int_{\Omega}u=-\mu\int_{\Omega}u+(a+C_1\mu)|\Omega|\quad\mbox{for all}\ t\in(0, T_{\max}).
\end{equation*}
This yields
\begin{equation}\label{du2}
\int_{\Omega}u\leq\left\{\int_{\Omega}u_0, \frac{(a+C_1\mu)|\Omega|}{\mu}\right\}=:C_2\quad\mbox{for all}\ t\in(0, T_{\max}).
\end{equation}
The proof of (\ref{nabla vs}) is based on standard regularity arguments for the heat equation; for details, we refer the reader to \cite[Lemma 4.1]{Horstmann&Winkler-JDE-2005}. Now, we prove (\ref{nabla v2}). Using $-2\Delta v$ as a test function for the second equation in (\ref{KS}), integrating over $\Omega$ and applying Young's inequality, we have
\begin{eqnarray}\label{delta v1}
\frac{d}{dt}\int_{\Omega}|\nabla v|^2+2\int_{\Omega}|\Delta v|^2+2\int_{\Omega}|\nabla v|^2&=&-2\int_{\Omega} u\Delta v\nonumber\\
                                        &\leq&\frac{1}{2}\int_{\Omega}u^2+2\int_{\Omega}|\Delta v|^2
\end{eqnarray}
for all $t\in(0, T_{\max})$. Combining (\ref{du1}) and (\ref{delta v1}), we deduce that
\begin{eqnarray}\label{u+v}
\frac{d}{dt}\left(\int_{\Omega}u+2\mu\int_{\Omega}|\nabla v|^2\right)+4\mu\int_{\Omega}|\nabla v|^2\leq\mu\int_{\Omega}u^2+a|\Omega|-\mu\int_{\Omega}u^{\gamma}
\end{eqnarray}
for all $t\in(0, T_{\max})$. Adding the term $2\int_{\Omega}u$ in both side of (\ref{u+v}) and using (\ref{du2}) yields
\begin{eqnarray*}
\frac{d}{dt}\left(\int_{\Omega}u+2\mu\int_{\Omega}|\nabla v|^2\right)+2\left(\int_{\Omega}u+2\mu\int_{\Omega}|\nabla v|^2\right)\leq-\mu\int_{\Omega}(u^{\gamma}-u^2)+a|\Omega|+2C_2
\end{eqnarray*}
for all $t\in(0, T_{\max})$. This, along with $\gamma\geq2$ by our assumption, gives $C_3>0$ such that
\begin{eqnarray*}
\frac{d}{dt}\left(\int_{\Omega}u+2\mu\int_{\Omega}|\nabla v|^2\right)+2\left(\int_{\Omega}u+2\mu\int_{\Omega}|\nabla v|^2\right)\leq C_3,
\end{eqnarray*}
for all $t\in(0, T_{\max})$, whereby the proof is completed.
\end{proof}

\vskip 3mm

For $p\geq1$, $q\geq1$ and $s\geq1$, we define
\begin{equation}\label{theta1}
\theta_1:=\theta_1(p,q)=\frac{2(p+\gamma-1)}{\gamma+1-\alpha-2\beta},
\end{equation}
\begin{equation}\label{theta2}
\theta_2:=\theta_2(p,q)=\frac{2(q-1)(p+\gamma-1)}{p+\gamma-3},
\end{equation}
\begin{equation}\label{kappa i}
\kappa_i:=\kappa_i(p,q;s)=\frac{\frac{q}{s}-\frac{q}{\theta_i}}{\frac{q}{s}-(\frac{1}{2}-\frac{1}{n})}
\end{equation}
and
\begin{equation*}\label{f i}
f_i(p,q;s):=\frac{\theta_i}{q}\kappa_i(p,q;s)=\frac{\frac{\theta_i}{s}-1}{\frac{q}{s}-(\frac{1}{2}-\frac{1}{n})}
\end{equation*}
for $i=1,2$. Now let state the following results which will be needed in the proof of the boundedness of global solutions. The main idea of the proof is similar to the strategy introduced in \cite{Stinner-JDE-2015}. Since a new parameter $\gamma$ is involved, we prefer to give enough details for the convenience of the reader.
\begin{lem}\label{parameter lem}
Let $n\geq2$, $\gamma\geq1$, $\alpha\in\mathbb{R}$ and $\beta\in\mathbb{R}$. Then for sufficiently large $p>1$,

(i) if $\alpha+2\beta<\gamma-1+\frac{2}{n}$, we can choose $q>1$ such that
\begin{equation}\label{parameter1}
\kappa_i(p,q; \frac{n}{n-1})\in(0,1)\quad\mbox{and}\quad f_i(p,q; \frac{n}{n-1})<2\quad\mbox{for}\ i=1,2;
\end{equation}

(ii) if $\alpha+2\beta<\gamma-1+\frac{4}{n+2}$, there exists $q>1$ fulfilling
\begin{equation}\label{parameter2}
\kappa_i(p,q; 2)\in(0,1)\quad\mbox{and}\quad f_i(p,q; 2)<2\quad\mbox{for}\ i=1,2.
\end{equation}
\end{lem}
\begin{rem}
Since in the case $1\leq\gamma<2$, we only obtain $\|\nabla v\|_{L^s(\Omega)}\leq C$ with $s\in[1,\frac{n}{n-1})$ according to Lemma \ref{basic est.}. We need to remark here if (\ref{parameter1}) is true, by continuity argument we can choose $s$ sufficiently close to $\frac{n}{n-1}$ satisfies $\kappa_i(p,q; s)\in(0,1)$ and $f_i(p,q; s)<2$ ($i=1,2$) when $p$ and $q$ are fixed. Hence it is enough to focus on the case $s=\frac{n}{n-1}$.
\end{rem}

\noindent{\it{Proof of Lemma \ref{parameter lem}}.}
We first claim that for $i=1,2$ if
\begin{equation}\label{claim}
\theta_i>s\quad\mbox{and}\quad q>\frac{\theta_i}{2}-\frac{s}{n},
\end{equation}
then
\begin{equation*}
\kappa_i(p,q; s)\in(0,1)\quad\mbox{and}\quad f_i(p,q; s)<2.
\end{equation*}
Indeed, a direct computation shows that the first inequality in (\ref{claim}) is equivalent to $\kappa_i(p,q; s)>0$. On the other hand, $\kappa_i(p,q; s)<1$ is equivalent to
\begin{equation*}
q>\frac{\theta_i}{2}-\frac{\theta_i}{n},
\end{equation*}
which is weaker then the second inequality in (\ref{claim}) because of $\theta_i>s$. Moreover, we have $f_i(p,q; s)<2$ if and only if
\begin{equation*}
\frac{\theta_i}{s}-1<\frac{2q}{s}-2(\frac{1}{2}-\frac{1}{n}),
\end{equation*}
which is true due to the second inequality in (\ref{claim}). The claim is proved.

Now we consider the case $s=\frac{n}{n-1}$. Note that (\ref{claim}) is fulfilled for $\theta_1$ and $\theta_2$ if
\begin{equation}\label{claim11}
p>\frac{n(\gamma+1-\alpha-2\beta)}{2(n-1)}-(\gamma-1),\qquad  q>\frac{n}{2(n-1)}+1
\end{equation}
and
\begin{equation}\label{claim12}
\frac{p+\gamma-1}{\gamma+1-\alpha-2\beta}-\frac{1}{n-1}<q<\frac{n}{2(n-1)}p+\frac{n}{2(n-1)}(\gamma-1)-\frac{1}{n-1}.
\end{equation}
One can easily remark that $q$ exists if
\begin{equation*}
\frac{p+\gamma-1}{\gamma+1-\alpha-2\beta}<\frac{n}{2(n-1)}p+\frac{n}{2(n-1)}(\gamma-1),
\end{equation*}
which can be achieved when $\alpha+2\beta<\gamma-1+\frac{2}{n}$. Fix
\begin{equation*}
p_0:=\max\left\{1, 3-\gamma, 2-\alpha-2\beta, \frac{3n(\gamma+1-\alpha-2\beta)}{2(n-1)}-(\gamma-1)\right\}.
\end{equation*}
Then, for any $p>p_0$, we can choose $q$ such that $p$ and $q$ satisfy (\ref{claim11}) and (\ref{claim12}). Hence we have (\ref{parameter1}).

When $s=2$, (\ref{claim}) is satisfied for $\theta_1$ and $\theta_2$ if
\begin{equation}\label{claim21}
p>2-\alpha-2\beta,\qquad  q>2
\end{equation}
and
\begin{equation}\label{claim22}
\frac{p+\gamma-1}{\gamma+1-\alpha-2\beta}-\frac{2}{n}<q<\frac{n+2}{2n}p+\frac{n+2}{2n}(\gamma-1)-\frac{2}{n}.
\end{equation}
Set
\begin{equation*}
\bar{p}_0:=\max\left\{1, 3-\gamma, 2-\alpha-2\beta, \frac{(2n+2)(\gamma+1-\alpha-2\beta)}{n}-(\gamma-1)\right\}.
\end{equation*}
Following in the same way as in (\ref{parameter1}), we have for arbitrary $p>\bar{p}_0$, there exists $q$ such that $p$ and $q$ fulfill (\ref{claim21}) and (\ref{claim22}). This completes the proof.$\hfill\Box$

\section{Proof of Theorem \ref{main result}}
According to test-function arguments and interpolation arguments along with the basic estimates in Lemma \ref{basic est.}, we have the boundedness of $\int_{\Omega}u^p$ with $p\geq1$.

\begin{lem}\label{p-q est.}
Suppose that $\Omega\subset\mathbb{R}^n$ ($n\geq2$). Let (\ref{D})-(\ref{logistic source}) hold with $M_1>0$, $M_2>0$, $\alpha\in\mathbb{R}$, $\beta\in\mathbb{R}$, $a\geq0$, $\mu>0$ and $\gamma\geq1$. Assume that
\begin{eqnarray*}
  \alpha+2\beta<\left\{\begin{array}{lll}
     \medskip
     \gamma-1+\frac{2}{n},&{} \mbox{if}\ 1\leq\gamma<2,\\
     \medskip
     \gamma-1+\frac{4}{n+2},&{} \mbox{if}\ \gamma\geq2.
  \end{array}\right.
\end{eqnarray*}
Then for all $p\in[1,\infty)$ there exists $C>0$ such that
\begin{equation*}
\|u(\cdot,t)\|_{L^p(\Omega)}\leq C\quad\mbox{for all}\ t\in(0, T_{\max}).
\end{equation*}
\end{lem}
\begin{proof}
Multiplying the first equation in (\ref{KS}) by the test function $(u+1)^{p-1}$, integrating by parts and using (\ref{logistic source}), we obtain
\begin{eqnarray}\label{test1}
&&\frac{1}{p}\frac{d}{dt}\int_{\Omega}(u+1)^p+(p-1)\int_{\Omega}(u+1)^{p-2}D(u)|\nabla u|^2+\frac{\mu}{2^{\gamma}}\int_{\Omega}(u+1)^{p+\gamma-1}\nonumber\\
&&\quad\leq(p-1)\int_{\Omega}(u+1)^{p-2}S(u)\nabla u\cdot\nabla v+(a+\mu)\int_{\Omega}(u+1)^{p-1}\quad\mbox{for all}\ t\in(0, T_{\max})
\end{eqnarray}
Here, by (\ref{D}),
\begin{equation}\label{test11}
(p-1)\int_{\Omega}(u+1)^{p-2}D(u)|\nabla u|^2\geq M_1(p-1)\int_{\Omega}(u+1)^{p-\alpha-2}|\nabla u|^2\quad\mbox{for all}\ t\in(0, T_{\max}),
\end{equation}
from (\ref{S}) and Young's inequality,
\begin{eqnarray}\label{test12}
&&(p-1)\int_{\Omega}(u+1)^{p-2}S(u)\nabla u\cdot\nabla v\nonumber\\
&&\quad\leq\frac{M_1(p-1)}{2}\int_{\Omega}(u+1)^{p-\alpha-2}|\nabla u|^2+\frac{p-1}{2M_1}\int_{\Omega}(u+1)^{p+\alpha-2}S^2(u)|\nabla v|^2\nonumber\\
&&\quad\leq\frac{M_1(p-1)}{2}\int_{\Omega}(u+1)^{p-\alpha-2}|\nabla u|^2+\frac{M^2_2(p-1)}{2M_1}\int_{\Omega}(u+1)^{p+\alpha+2\beta-2}|\nabla v|^2
\end{eqnarray}
for all $t\in(0, T_{\max})$, and again using Young's inequality,
\begin{equation}\label{test13}
(a+\mu)\int_{\Omega}(u+1)^{p-1}\leq\frac{\mu}{2^{\gamma+1}}\int_{\Omega}(u+1)^{p+\gamma-1}+C_1\quad\mbox{for all}\ t\in(0, T_{\max})
\end{equation}
with $C_1>0$. Inserting (\ref{test11})-(\ref{test13}) into (\ref{test1}) yields
\begin{eqnarray}\label{d-test1}
&&\frac{1}{p}\frac{d}{dt}\int_{\Omega}(u+1)^p+\frac{2M_1(p-1)}{(p-\alpha)^2}\int_{\Omega}|\nabla(u+1)^{\frac{p-\alpha}{2}}|^2
+\frac{\mu}{2^{\gamma+1}}\int_{\Omega}(u+1)^{p+\gamma-1}\nonumber\\
&&\quad\leq\frac{M^2_2(p-1)}{2M_1}\int_{\Omega}(u+1)^{p+\alpha+2\beta-2}|\nabla v|^2+C_1\quad\mbox{for all}\ t\in(0, T_{\max}).
\end{eqnarray}
Next, differentiating the second equation in (\ref{KS}) and using the identity
\begin{equation*}
\frac{1}{2}\Delta|\nabla v|^2=\nabla(\Delta v)\cdot\nabla v+|D^2v|^2,
\end{equation*}
where $|D^2v|^2=\sum_{1\leq i,j\leq n}(v_{x_ix_j})^2$, we obtain
\begin{equation*}
\frac{1}{2}\left(|\nabla v|^2\right)_t=\frac{1}{2}\Delta|\nabla v|^2-|D^2v|^2-|\nabla v|^2+\nabla u\cdot\nabla v
\end{equation*}
for all $(x,t)\in\Omega\times(0, T_{\max})$. We test this by $(|\nabla v|^2)^{q-1}$ and integrate over $\Omega$ to have
\begin{eqnarray}\label{test2}
&&\frac{1}{2q}\frac{d}{dt}\int_{\Omega}|\nabla v|^{2q}+\frac{q-1}{2}\int_{\Omega}|\nabla v|^{2q-4}\left|\nabla|\nabla v|^2\right|^2+\int_{\Omega}|\nabla v|^{2q}+\int_{\Omega}|\nabla v|^{2q-2}|D^2v|^2\nonumber\\
&&\quad=\int_{\Omega}|\nabla v|^{2q-2}\nabla u\cdot\nabla v+\frac{1}{2}\int_{\partial\Omega}|\nabla v|^{2q-2}\frac{\partial|\nabla v|^2}{\partial\nu}d{\bf{S}}\quad\mbox{for all}\ t\in(0, T_{\max}).
\end{eqnarray}
Integrating by parts and twice applying Young's inequality, we estimate
\begin{eqnarray}\label{test21}
\int_{\Omega}|\nabla v|^{2q-2}\nabla u\cdot\nabla v&=&-(q-1)\int_{\Omega}u|\nabla v|^{2q-4}\nabla|\nabla v|^2\cdot\nabla v-\int_{\Omega}u|\nabla v|^{2q-2}\Delta v\nonumber\\
                                                   &\leq&\frac{q-1}{12}\int_{\Omega}|\nabla v|^{2q-4}\left|\nabla|\nabla v|^2\right|^2
                                                         +3(q-1)\int_{\Omega}u^2|\nabla v|^{2q-2}\nonumber\\
                                                   &\quad&+\int_{\Omega}|\nabla v|^{2q-2}|D^2v|^2
                                                   +\frac{n}{4}\int_{\Omega}u^2|\nabla v|^{2q-2}
\end{eqnarray}
for all $t\in(0, T_{\max})$, where we have used that $\frac{1}{n}|\Delta v|^2\leq |D^2v|^2$. As for the second term on the right-hand side of (\ref{test2}), we use the inequality \cite[Lemma 4.2]{MS}
\begin{equation*}
\frac{\partial|\nabla w|^2}{\partial\nu}\leq2k|\nabla w|^2\quad\mbox{on}\ \partial\Omega,
\end{equation*}
with $k=k(\Omega)>0$ is an upper bound of the curvature of $\partial\Omega$, and apply the trace inequality \cite[Remark 52.9]{QS}
\begin{equation*}
\|w\|_{L^2(\partial\Omega)}\leq\eta\|\nabla w\|_{L^2(\Omega)}+C(\eta)\|w\|_{L^2(\Omega)}
\end{equation*}
to deduce that
\begin{eqnarray}\label{test22}
\frac{1}{2}\int_{\partial\Omega}|\nabla v|^{2q-2}\frac{\partial|\nabla v|^2}{\partial\nu}d{\bf{S}}&=&\frac{1}{2q}\int_{\partial\Omega}\frac{\partial\left(|\nabla v|^q\right)^2}{\partial\nu}d{\bf{S}}\nonumber\\
                           &\leq&\frac{k}{q}\left\||\nabla v|^q\right\|^2_{L^2(\partial\Omega)}\nonumber\\
                           &\leq&\frac{q-1}{3q^2}\left\|\nabla|\nabla v|^q\right\|^2_{L^2(\Omega)}+C_2\left\||\nabla v|^q\right\|^2_{L^2(\Omega)}
\end{eqnarray}
for all $t\in(0, T_{\max})$ with some $C_2>0$. From the Gagliardo-Nirenberg inequality, we know that there exist $C_3>0$ and $C_4>0$ such that
\begin{equation*}
\left\||\nabla v|^q\right\|^2_{L^2(\Omega)}\leq C_3\left\|\nabla|\nabla v|^q\right\|^{2\kappa}_{L^2(\Omega)}\left\||\nabla v|^q\right\|^{2(1-\kappa)}_{L^{\frac{s}{q}}(\Omega)}+C_4\left\||\nabla v|^q\right\|^{2}_{L^{\frac{s}{q}}(\Omega)}
\end{equation*}
for all $t\in(0, T_{\max})$, where
\begin{equation*}
\kappa=\frac{\frac{q}{s}-\frac{1}{2}}{\frac{q}{s}-\frac{1}{2}+\frac{1}{n}}\in(0,1)
\end{equation*}
with $s\in[1, \frac{n}{n-1})$. According to (\ref{nabla vs}) as well as Young's inequality, we have
\begin{equation}\label{GN1}
C_2\left\||\nabla v|^q\right\|^2_{L^2(\Omega)}\leq \frac{q-1}{3q^2}\left\|\nabla|\nabla v|^q\right\|^{2}_{L^2(\Omega)}+C_5\quad\mbox{for all}\ t\in(0, T_{\max}).
\end{equation}
Now collecting (\ref{test2})-(\ref{GN1}),  we get
\begin{eqnarray}\label{d-test2}
&&\frac{1}{2q}\frac{d}{dt}\int_{\Omega}|\nabla v|^{2q}+\frac{q-1}{q^2}\left\|\nabla|\nabla v|^q\right\|^{2}_{L^2(\Omega)}+\int_{\Omega}|\nabla v|^{2q}\nonumber\\
&&\quad\leq\left(3(q-1)+\frac{n}{4}\right)\int_{\Omega}u^2|\nabla v|^{2q-2}+C_5\quad\mbox{for all}\ t\in(0, T_{\max}).
\end{eqnarray}
Combining (\ref{d-test1}) with (\ref{d-test2}), it follows that
\begin{eqnarray*}
&&\frac{d}{dt}\left(\frac{1}{p}\int_{\Omega}(u+1)^p+\frac{1}{2q}\int_{\Omega}|\nabla v|^{2q}\right)\nonumber\\
&&\quad\quad+\frac{2M_1(p-1)}{(p-\alpha)^2}\int_{\Omega}|\nabla(u+1)^{\frac{p-\alpha}{2}}|^2+\frac{q-1}{q^2}
             \left\|\nabla|\nabla v|^q\right\|^{2}_{L^2(\Omega)}\nonumber\\
&&\quad\quad+\frac{\mu}{2^{\gamma+1}}\int_{\Omega}(u+1)^{p+\gamma-1}+\int_{\Omega}|\nabla v|^{2q}\nonumber\\
&&\quad\leq C_6\int_{\Omega}(u+1)^{p+\alpha+2\beta-2}|\nabla v|^2+C_7\int_{\Omega}(u+1)^2|\nabla v|^{2q-2}+C_1+C_5
\end{eqnarray*}
for all $t\in(0, T_{\max})$ with certain positive constants $C_6$ and $C_7$. Since $\alpha+2\beta<\gamma+1$, by Young's inequality we can find
$C_8$ and $C_9$ such that
\begin{eqnarray}\label{d-test1+d-test2}
&&\frac{d}{dt}\left(\frac{1}{p}\int_{\Omega}(u+1)^p+\frac{1}{2q}\int_{\Omega}|\nabla v|^{2q}\right)\nonumber\\
&&\quad\quad+\frac{2M_1(p-1)}{(p-\alpha)^2}\int_{\Omega}|\nabla(u+1)^{\frac{p-\alpha}{2}}|^2+\frac{q-1}{q^2}
             \left\|\nabla|\nabla v|^q\right\|^{2}_{L^2(\Omega)}\nonumber\\
&&\quad\quad+\frac{\mu}{2^{\gamma+1}}\int_{\Omega}(u+1)^{p+\gamma-1}+\int_{\Omega}|\nabla v|^{2q}\nonumber\\
&&\quad\leq \frac{\mu}{2^{\gamma+2}}\int_{\Omega}(u+1)^{p+\gamma-1}
+C_8\int_{\Omega}|\nabla v|^{\theta_1}+C_8\int_{\Omega}|\nabla v|^{\theta_2}+C_{9},
\end{eqnarray}
where $\theta_1$ and $\theta_2$ are given by (\ref{theta1}) and (\ref{theta2}), respectively. According to the Gagliardo-Nirenberg inequality, for $i=1,2$ we can pick $C_{10}>0$ such that
\begin{eqnarray}\label{GNi}
C_8\int_{\Omega}|\nabla v|^{\theta_i}&=&\left\||\nabla v|^q\right\|^{\frac{\theta_i}{q}}_{L^{\frac{\theta_i}{q}}(\Omega)}\nonumber\\
&\leq& C_{10}\left(\left\|\nabla|\nabla v|^q\right\|^{\kappa_i}_{L^2(\Omega)}
      \left\||\nabla v|^q\right\|^{1-\kappa_i}_{L^{\frac{s}{q}}(\Omega)}
      +\left\||\nabla v|^q\right\|_{L^{\frac{s}{q}}(\Omega)}\right)^{\frac{\theta_i}{q}}
\end{eqnarray}
for all $t\in(0, T_{\max})$, where $\kappa_i$ is defined by (\ref{kappa i}) . If $1<\gamma<2$, we choose $s=\frac{n}{n-1}$ in (\ref{GNi}). Due to (\ref{nabla vs}), (\ref{parameter1}) and the Young's inequality show that with some $C_{11}>0$ we have for $i=1,2$
\begin{eqnarray}\label{GN1}
C_8\int_{\Omega}|\nabla v|^{\theta_i}\leq\frac{q-1}{4q^2}\left\|\nabla|\nabla v|^q\right\|^{2}_{L^2(\Omega)}+C_{11}\quad\mbox{for all}\ t\in(0, T_{\max}).
\end{eqnarray}
If $\gamma\geq2$, we set $s=2$ in (\ref{GNi}). Then in view of (\ref{nabla v2}), (\ref{parameter2}) and the Young's inequality, we obtain $C_{12}$ such that
\begin{eqnarray}\label{GN2}
C_8\int_{\Omega}|\nabla v|^{\theta_i}\leq\frac{q-1}{4q^2}\left\|\nabla|\nabla v|^q\right\|^{2}_{L^2(\Omega)}+C_{12}\quad\mbox{for all}\ t\in(0, T_{\max})
\end{eqnarray}
for $i=1,2$. Set
\begin{equation*}
y(t):=\frac{1}{p}\int_{\Omega}(u+1)^p+\frac{1}{2q}\int_{\Omega}|\nabla v|^{2q}.
\end{equation*}
According to Young's inequality we can find $C_{13}$ such that
\begin{equation}\label{Young}
\frac{\mu}{2^{\gamma+2}}\int_{\Omega}(u+1)^{p+\gamma-1}\geq\int_{\Omega}(u+1)^{p}-C_{13}\quad\mbox{for all}\ t\in(0, T_{\max})
\end{equation}
and therefore in conjunction with (\ref{d-test1+d-test2}), (\ref{GN1}) and (\ref{GN2}) and (\ref{Young}) we see that
\begin{equation}\label{ODI}
y'(t)+C_{14}y(t)\leq C_{15}\quad\mbox{for all}\ t\in(0, T_{\max})
\end{equation}
with positive constants $C_{14}$ and $C_{15}$. This completes the proof.
\end{proof}

We can now pass to the proof of our main result.

\noindent{\it{Proof of Theorem \ref{main result}}.} With the aid of Lemma \ref{p-q est.}, \cite[Lemma A.1]{Tao&Winkler-JDE-2012} and Lemma \ref{local existence}, we obtain the boundedness of $u$. Then the boundedness of $v$ is based on a standard argument involving the variation-of-constants representation for $v$ and the smoothing estimates for heat semigroup.$\hfill\Box$
%\noindent\emph{Acknowledgements.}

%%%%%%%%%%%%%%%%%%%%%%%%%%%%%%%%%%%%%%%%%%%%%%%%%%%%%%%%%%%%%%%%%%%%%%%%%%%%%%%%

%\bibliographystyle{AIMS}
%\bibliography{Ref}

\end{document}